\documentclass[a4paper,11pt]{article}

\pdfoutput=1

\usepackage{amsmath,amsthm,amssymb,mathtools}
\mathtoolsset{showonlyrefs=true}
\usepackage[a4paper]{geometry}
\usepackage{graphicx}
\usepackage{microtype}
\usepackage{enumitem}
\usepackage{color}
\usepackage{relsize}
\newcommand{\K}[1]{\mathbb{#1}}
\newcommand{\txt}[1]{\text{\normalfont{#1}}}

\newcommand{\C}[1]{\mathcal{#1}}
\newcommand{\RP}{\mathbb{R}\textrm{\normalfont{P}}^1}
\newcommand{\mean}{\mathbb{E}}
\newcommand{\G}[2]{\Gamma\left(\frac{#1}{#2}\right)}

\newcommand\blfootnote[1]{%
  \begingroup
  \renewcommand\thefootnote{}\footnote{#1}%
  \addtocounter{footnote}{-1}%
  \endgroup
}

\newtheorem{thm}{Theorem}
\newtheorem*{thm*}{Theorem}
\newtheorem{defi}{Definition}
\newtheorem{remark}{Remark}
\newtheorem{propo}{Proposition}
\newtheorem{lemma}{Lemma}
\newtheorem{cor}{Corollary}

\makeatletter
\newenvironment{proofof}[1]{\par
  \pushQED{\qed}%
  \normalfont \topsep6\p@\@plus6\p@\relax
  \trivlist
  \item[\hskip\labelsep
    \itshape{Proof of #1\@addpunct{.}}]\ignorespaces
}{%
  \popQED\endtrivlist\@endpefalse
}
\makeatother

\title{The real polynomial eigenvalue problem is well conditioned on the average}
\author{Carlos Beltr\'an and Khazhgali Kozhasov}

\date{}

\begin{document}
\maketitle

\begin{changemargin}{1.5cm}{1.5cm}
{\bf\noindent Abstract.}
We study the average condition number for polynomial eigenvalues of collections of matrices drawn from various random matrix ensembles. In particular, we prove that polynomial eigenvalue problems defined by matrices with Gaussian entries are very well-conditioned on the average.
\end{changemargin}
\blfootnote{Carlos Beltr\'an is partially supported by grants MTM2017-83816-P and MTM2014-57590-P from Spanish Ministerio de Econom\'ia y Competitividad, as well as by grant SI01.64658 from Banco Santander and Universidad de Cantabria.}
\blfootnote{\textup{2010} \textit{Mathematics Subject Classification}: \textup{14Q20, 15A18, 15A22, 15B52, 65F15}} 

\section*{Introduction}
Following the ideas in \cite{SS1,BCSS}, we note that many different numerical problems can be described within the following simple general framework. We consider a space of {\em inputs} and a space of {\em outputs} denoted by $\mathcal{I}$ and $\mathcal{O}$ respectively, and some equation of the form $ev(i,o)=0$ stating when an output is a solution for a given input. Both $\mathcal{I}$ and $\mathcal{O}$, and the {\em solution variety}
\[
\mathcal{V}=\{(i,o)\in\mathcal{I}\times\mathcal{O}:\text{$o$ is an output to $i$}\}=\{(i,o)\in\mathcal{I}\times\mathcal{O}: ev(i,o)=0\}
\]
are frequently real algebraic or just semialgebraic sets. The numerical problem to be solved can then be written as ``given $i\in\mathcal{I}$, find $o\in\mathcal{O}$ such that $(i,o)\in\mathcal{V}$'', or ``find all $o\in\mathcal{O}$ such that $(i,o)\in\mathcal{V}$''. One can have in mind the following examples:
\begin{enumerate}
	\item Polynomial Root Finding: $\mathcal{I}$ is the set of univariate real polynomials of degree $d$, $\mathcal{O}=\K{R}$ and $\mathcal{V}=\{(f,\zeta):f(\zeta)=0\}$.
	\item Polynomial System Solving, which we can see as the homogeneous multivariate version of Polynomial Root Finding: $\mathcal{I}$ is the projective space of (dense or structured) systems of $n$ real homogeneous polynomials of degrees $d_1,\ldots,d_n$ in variables $x_0,\ldots,x_n$, $\mathcal{O}=\K{R}\txt{P}^n$ and $\mathcal{V}=\{(f,\zeta):f(\zeta)=0\}$.
	\item EigenValue Problem: $\mathcal{I}=\K{R}^{n\times n}$, $\mathcal{O}=\K{R}$ and  $\mathcal{V}=\{(A,\lambda):\txt{det}(A-\lambda\,\txt{Id})= 0\}$.
	\item (Homogeneous) Polynomial EigenValue Problem (in the sequel called PEVP): $\mathcal{I}$~is~the set of tuples of $d+1$ real $n\times n$ matrices $A=(A_0,\ldots,A_d)$, $\mathcal{O}=\K{R}\txt{P}^1$ and $\mathcal{V}=\{(A,[\alpha:\beta]):P(A,\alpha,\beta)=\det(\alpha^0\beta^dA_0+\alpha^1\beta^{d-1} A_1+\cdots+\alpha^d\beta^0A_d)=0\}$. One~can force some of the matrices to be symmetric, a particularly important case in applications, or consider other structured problems, see \cite{MehrmannVoss,DedieuTisseur,Tisseur,TisseurMeerbergen}. In cases $d=1$ and $d=2$ polynomial eigenvalues are often referred to as \emph{generalized eigenvalues} and \emph{quadratic eigenvalues} respectively. 
\end{enumerate}

In this paper we prove a general theorem computing exactly the expected value of the condition number in a wide collection of problems, including problem 4 above.

We start by recalling the general geometric definition of the condition number, which is usually thought of as ``a measure of the sensibility of the solution $o$ under an infinitesimal perturbation of the input $i$''. A \emph{Finsler structure} on a differentiable manifold $M$ is a smooth field of norms $\Vert \cdot \Vert_p: T_pM \rightarrow \K{R},\,p\in M$ on $M$ (see \cite[p. 223]{BCSS} for more details). In particular, a Riemannian structure $\langle \cdot,\cdot \rangle$ on $M$ defines a Finsler structure on it by  $\|\dot p\|_p=\sqrt{\langle \dot p,\dot p\rangle_p}$, $p\in M$, $\dot p\in T_pM$.
\begin{defi}[Condition number in the algebraic setting]\label{def:cn}
Let $\mathcal{I},\mathcal{O}$ and $\mathcal{V}$ be real algebraic varieties such that the smooth loci of $\mathcal{I},\mathcal{O}$ are endowed with Finsler structures and let $(i,o)\in \C{V}$ be a smooth point of $\C{V}$ such that $i\in \C{I}, o\in \C{O}$ are smooth points of $\C{I}$ and $\C{O}$ respectively. Moreover, assume that $D_{(i,o)}p_1: T_{(i,o)}\C{V} \rightarrow T_{i}\C{I}$ is invertible. Then the \emph{condition number} $\mu(i,o)$ of $(i,o)\in \C{V}$ is defined as
\[
\mu(i,o)=\left\|D_{(i,o)}p_2\circ D_{(i,o)}p_1^{-1}\right\|_{\txt{op}},
\]
where $p_1:\mathcal{V}\to\mathcal{I},\, p_2:\mathcal{V}\to\mathcal{O}$ are the projections and $\|\cdot\|_{\txt{op}}$ is the operator norm. For points $(i,o)\in \C{V}$ not satisfying the above assumptions the condition number is set to $\infty$.
\end{defi}
See \cite[Sec. 14.1]{BuCu} for more on this geometric approach to the condition number.
\begin{remark}
Definition \ref{def:cn} is intrinsic in $\C{I}$, i.e., changing $\C{I}$ to some subvariety $\C{I}^{\prime}\subset \C{I}$ leads (in general) to different, smaller, value of the condition number, since perturbations of the input are only allowed in the direction of the tangent space to the input set. Note also that the condition number depends on choices of Finsler structures on $\C{I}$ and $\C{O}$.
\end{remark}

\noindent
\textit{Example:}
The classical Turing's condition number $\mu(A) = \Vert A\Vert_{\txt{op}} \Vert A^{-1}\Vert_{\txt{op}}$ for matrix inversion corresponds to the following setting:
\begin{itemize}
\item $\C{O} =\C{I}=M(n,\K{R})$ is the set of $n\times n$ real  matrices endowed with the Finsler structure associated to relative errors in operator norm: $\|\dot A\|_A=\|\dot A\|_{\txt{op}}/\|A\|_{\txt{op}}$.
\item $\C{V}=\{(A,B):AB=\txt{Id}\} = \{(A,B): B=A^{-1}\}$. 
\end{itemize}

In the PEVP the input space $\C{I}$ is endowed with the following Riemannian structure: $\langle \dot A,\dot B\rangle_A = ((\dot A_0,\dot B_0) + \dots + (\dot A_d,\dot B_d))/((A_0,A_0)+\dots+(A_d,A_d))$, where $(\cdot,\cdot)$ is the Frobenius inner product, $A= (A_0,\dots,A_d)$ and $\dot A=(\dot A_0,\dots,\dot A_d), \dot B= (\dot B_0,\dots,\dot B_d)\in T_{A}\C{I}$. The output space $\C{O}=\RP$ possesses the standard metric and the solution variety $\C{V}= \{(A,[\alpha:\beta]) : P(A,\alpha,\beta)=0\}$ is endowed with the induced product Riemannian structure.
An explicit formula for the condition number for the Homogeneous PEVP was derived in \cite[Th. 4.2]{DeTi} (we write here the relative condition number version):
\begin{equation}\label{eq:evc}
\mu(A,(\alpha,\beta))=\left(\sum_{k=0}^d \alpha^{2k} \beta^{2d-2k}\right)^{1/2}\frac{\Vert r\Vert \Vert \ell\Vert}{|\ell^t v|}\,\|A\|,
\end{equation}
where $A=(A_0,\ldots,A_d)$, $(\alpha,\beta)\in \K{R}^2$ is a polynomial eigenvalue of $A$, $r$ and $\ell$ are the {corresponding} right and left eigenvectors and
\begin{equation}\label{eq:v}
v=\beta \frac{\partial}{\partial \alpha} P( A,\alpha,\beta)r-\alpha \frac{\partial}{\partial \beta}P( A,\alpha,\beta)r.
\end{equation}
A given tuple $A$ can have up to $nd$ real isolated polynomial eigenvalues. We define the condition number of $A$ simply as the sum of the condition numbers over all these PEVs:
\[
\mu(A) =\sum_{[\alpha:\beta]\in \RP \txt{is a PEV of } A} \mu(A,(\alpha,\beta)).
\]
(If $A=(A_0,\dots,A_d)$ has infinitely many polynomial eigenvalues, then we set $\mu(A)=\infty$). The most important result in this paper is a very general theorem which is designed to provide exact formulas for the expected value of the condition number in the PEVP and other problems. A simple particular case of our general theorem is as follows.
\begin{thm}[Gaussian Homogeneous PEVP are well conditioned on the average]\label{gaussian}
	If $A_0,\dots,A_d\in \C{N}_{M(n,\K{R})}$ are independent $\C{N}_{M(n,\K{R})}$-distributed matrices, then
	\begin{align}\label{eq:gaussian}
	\underset{A_0,\dots,A_d\sim i.i.d.\, \mathcal{N}_{M(n,\K{R})}}{\mean}\, \mu(A) &= \pi \frac{\Gamma\left(\frac{(d+1)n^2}{2}\right)}{\Gamma\left(\frac{(d+1)n^2-1}{2}\right)} \frac{\G{n+1}{2}}{\G{n}{2}} \\
	&= \frac{\pi}{2} \sqrt{(d+1)n^3}\left(1 + \C{O}\left(\frac{1}{n}\right)\right),\ n\rightarrow +\infty
	\end{align}
\end{thm}
In Corollary \ref{eq:GOE} we provide an analogous formula in the case when $A_0,\ldots,A_d$ are independent GOE$(n)$-distributed matrices.
\begin{remark}
	Recently in \cite{ArmBel} Armentano and the first author of the current article investigated the expectation of the squared condition number for polynomial eigenvalues of complex Gaussian matrices. Theorem \ref{gaussian} establishes the ``asymptotic square root law'' for the considered problem, i.e., when $n\rightarrow +\infty$ (and up to the factor $\pi/2$) our answer in \eqref{eq:gaussian} equals the square root of the answer in \cite{ArmBel}.
\end{remark}
In Section \ref{sec:main} we state our main results, of which Theorem \ref{gaussian} is an easy consequence. Their proofs are given in Section \ref{main} and in Section \ref{applications}, some technical results are left for the Appendix.

\section{Main results}\label{sec:main}

In this section we state our most general result, from which Theorem \ref{gaussian} will follow. First, let us fix a general framework which analyzes the input-output problems described above in a semialgebraic context. For the rest of this paper the input and the output sets will be, respectively, the real vector space $\C{I}=\K{R}^m$ and the unit circle $S^1\subset \K{R}^2$ endowed with the standard Riemannian structures. The solution variety will be a semialgebraic set $\C{S}\subset \K{R}^m\times S^1\subset \K{R}^m\times \K{R}^2$ (we change letter from $\C{V}$ to $\C{S}$ to remark the fact that it is semialgebraic). We denote by $\C{S}_{\txt{top}}$ the union of top-dimensional (smooth) strata of $\C{S}$ (see Section \ref{preliminaries} for details). Then the smooth manifold $\C{S}_{\txt{top}}\subset \K{R}^m\times S^1$ is endowed with the induced Riemannian structure. The two projections defined on $\C{S}$ are denoted by $p_1: \C{S} \rightarrow \K{R}^m,\,p_2: \C{S} \rightarrow S^1$.

\begin{defi}[Condition number in the semialgebraic setting]\label{condition}
 Near a regular point $(a,x)\in \C{S}_{\txt{top}}$ the first projection $p_1: \C{S}_{\txt{top}} \rightarrow \K{R}^m$ is locally invertible, i.e., there exists a neighbourhood $U\subset \K{R}^m$ of $a\in U$ and a unique smooth map $p_1^{-1}: U \rightarrow \C{S}_{\txt{top}}$ such that $p_1^{-1}(a) = (a,x)$ and $p_1\circ p_1^{-1} = \txt{id}_{U}$. In this case the \emph{local relative condition number} $\mu(a,x)$ is defined as
\begin{align}
  \mu(a,x):= \Vert a\Vert \sup\limits_{\dot a \in \K{R}^m\setminus \{0\}} \frac{\Vert D_a(p_2\circ p^{-1}_1)(\dot a)\Vert}{\Vert\dot a\Vert}
\end{align}
For points $(a,x)\in \C{S}_{\txt{low}} = \C{S}\setminus \C{S}_{\txt{top}}$ in the strata of lower dimension of $\C{S}$ as well as for critical points $(a,x)\in \C{S}_{\txt{top}}$ of $p_1:\C{S}_{\txt{top}} \rightarrow \K{R}^m$ we set $\mu(a,x):= \infty$. 

\emph{The relative condition number} $\mu(a)$ of $a\in \K{R}^m$ is defined to be the sum of all local relative condition numbers $\mu(a,x)$:
\begin{align}
\mu(a) := \sum\limits_{x\in S^1:\, (a,x)\in \C{S}} \mu(a,x)
\end{align}
\end{defi}
\begin{remark}
Note that Definition \ref{condition} agrees with Definition \ref{def:cn} if we endow the input space $\C{I}=\K{R}^m$ with the Riemannian structure associated to relative errors, that is $\langle \dot a,\dot b\rangle_a=(\dot b^{\,t}\dot a)/\|a\|^2,\, a\in \K{R}^m$.
\end{remark}
To simplify terminology, throughout the rest of the paper, we omit the word ``relative'' when refering to (local) relative condition number.

We deal with a large class of semialgebraic subsets of $\K{R}^m\times S^1$ that we define next. 
\begin{defi}\label{def:S} We say that the semialgebraic set $\C{S}\subset \K{R}^m\times S^1$ is \emph{non-degenerate} if the following conditions are satisfied:
\begin{enumerate}\label{conditions}
\item[1.] 
for any $x\in S^1$ the fiber $p_2^{-1}(x)$ is of dimension $m-1$,
\item[2.]
the semialgebraic set $\Sigma_1\subset \C{S}_{\txt{top}}$ of critical points of $p_1: \C{S}_{\txt{top}} \rightarrow \K{R}^m$ is at most $(m-1)$-dimensional.
In Proposition \ref{equivalence} we show that this condition is equivalent to the following one: 
\item[$2^{\prime}\hspace{-2.5pt}$.] there exists a semialgebraic subset $B\subset\K{R}^m$ of dimension at most $m-2$ such that for any $a\notin B$ the fiber $p_1^{-1}(a)$ is finite.
\end{enumerate}
\end{defi}
The first condition in Definition \ref{def:S} implies that $\C{S}$ is $m$-dimensional (see Lemma \ref{lemma1}).
To perform our probabilistic study we take the input variables $a=(a_1,\dots,a_m)\in \K{R}^m$ to be independent standard gaussians: $a\sim N(0,1)$. In the following theorem we establish a general formula for the expectation of the condition number $\mu(a)$ of a randomly chosen $a\in \K{R}^m$:
\begin{thm}\label{general}
If $\C{S}\subset \K{R}^m\times S^1$ is a non-degenerate semialgebraic set, then
\begin{align}\label{general1}
  \mean_{a\sim N(0,1)} \left( \sum\limits_{x\in S^1 :\, (a,x)\in \C{S}} \mu(a,x) \right) = \frac{1}{\sqrt{2\pi}^m} \int\limits_{x\in S^1} \int\limits_{a\in p_2^{-1}(x)}\Vert a\Vert e^{-\frac{\Vert a\Vert^2}{2}} da\,dx.
\end{align}
If, moreover, $\C{S}$ is scale-invariant with respect to the first $m$ variables, i.e., $(a,x)\in \C{S}$ if and only if $(ta,x)\in \C{S}$ for any $t>0$, then
\begin{align}\label{general2}
  \mean_{a\sim N(0,1)} \left( \sum\limits_{x\in S^1 :\, (a,x)\in \C{S}} \mu(a,x) \right) = \frac{\Gamma\left(\frac{m}{2}\right)}{2\sqrt{\pi}^m} \int\limits_{x\in S^1} |p_2^{-1}(x) \cap S^{m-1}|\, dx,
\end{align}
 where $|p_2^{-1}(x)\cap S^{m-1}|$ denotes the volume of the $(m-2)$-dimensional semialgebraic spherical set $p_2^{-1}(x)\cap S^{m-1}$.
\end{thm}
The following form of Theorem \ref{general} for sets in $\K{R}^m\times \RP$ better fits our purposes.
\begin{cor}\label{cor:projective}
Let $\C{S}\subset \K{R}^m\times S^1$ be a non-degenerate semialgebraic set that is scale-invariant with respect to the first $m$ variables and suppose that $\C{S}$ is invariant under the map $(a,x)\mapsto (a,-x),\, (a,x)\in \K{R}^m\times S^1$. Then $\mu(a,x)=\mu(a,-x),\,(a,x)\in \C{S}$, the fibers $p_2^{-1}(x), p_2^{-1}(-x)$ are isometric and
\begin{align}\label{general2}
  \mean_{a\sim N(0,1)} \left( \sum\limits_{[x]\in \RP :\, (a,x)\in \C{S}} \mu(a,x) \right) = \frac{\Gamma\left(\frac{m}{2}\right)}{2\sqrt{\pi}^m} \int\limits_{[x]\in \RP} |p_2^{-1}(x) \cap S^{m-1}|\, d[x],
\end{align}
\end{cor}
Note that Corollary \ref{cor:projective} is just a ``projective'' version of the second part of Theorem \ref{general}.

As pointed out in the introduction, we are specifically interested in the \emph{polynomial eigenvalue problem}. Given $d+1$ matrices $A_0,\dots,A_d \in M(n,\K{R})$ a point $[x]=[\alpha:\beta]\in \RP$ is a (real)  \emph{polynomial eigenvalue} (PEV) of $A=(A_0,\dots,A_d)$ if  $$\txt{det}(\alpha^0\beta^d A_0 + \dots+\alpha^d\beta^0 A_d)=0.$$

The space $M(n,\K{R})$ of $n\times n$ real matrices is endowed with the Frobenius inner product and the associated norm: 
$$(A,B) = \txt{tr}(A^t B),\quad \Vert A\Vert^2 = (A,A),\quad A,B\in M(n,\K{R}).$$ Then a $k$-dimensional vector subspace $V\subset M(n,\K{R})$ is endowed with the standard normal probability distribution $\mathcal{N}_V$:
\begin{align}
\txt{P}_{\mathcal{N}_V} (U) = \frac{1}{\sqrt{2\pi}^{\,k}} \int\limits_U e^{-\frac{\Vert v\Vert^2}{2}}dv,
\end{align}
where $dv$ is the Lebesgue measure on $(V,(\cdot,\cdot))$ and $U\subset V$ is a measurable subset. Let us also denote by $\Sigma_V=\{A\in V: \txt{det}\, A = 0\}\subset V$ the variety of singular matrices in $V$.  

\emph{The condition number} for polynomial eigenvalues of $A=(A_0,\dots,A_d)\in V^{d+1}$ is defined via
\[
\mu(A) :=\sum_{[x]\in \RP \txt{is a PEV of } A} \mu(A,x),
\]
where $\mu(A,x)$ is as in Definition \ref{condition} with $\K{R}^m=(V,(\cdot,\cdot))^{d+1}$ and
\begin{align}
\C{S}= \{(A,x)=((A_0,\dots,A_d),(\alpha,\beta))\in V^{d+1}\times S^1 : \txt{det}(\alpha^0\beta^d A_0 + \dots+\alpha^d\beta^0 A_d)=0\}
\end{align}
As proved in \cite{DeTi}, in the case $V=M(n,\K{R})$ this definition for $\mu(A,x)$ is equivalent to \eqref{eq:evc}. 
In the following theorem we investigate the expected condition number for polynomial eigenvalues of independent $\mathcal{N}_V$-distributed matrices $A_0,\dots,A_d\in V$.
\begin{thm}\label{pevp general}
If $\Sigma_V\subset V$ is of codimension one, then
\begin{align}\label{eq:pevp general}
      \underset{A_0,\dots,A_d\sim i.i.d.\, \mathcal{N}_V}{\mean}\, \mu(A) =  \sqrt{\pi} \frac{\Gamma\left(\frac{(d+1)k}{2}\right)}{\Gamma\left(\frac{(d+1)k-1}{2}\right)}\frac{|\Sigma_V\cap S^{k-1}|}{|S^{k-2}|}.
    \end{align}
\end{thm}
Poincar\'e formula \cite[(3-5)]{Howard} allows to derive the following universal upper bound.
\begin{cor}\label{bound}
If $\Sigma_V\subset V$ is of codimension one, then
\begin{align}\label{eq:bound}
      \underset{A_0,\dots,A_d\sim i.i.d.\, \mathcal{N}_V}{\mean}\, \mu(A) \leq  \sqrt{\pi} n \frac{\Gamma\left(\frac{(d+1)k}{2}\right)}{\Gamma\left(\frac{(d+1)k-1}{2}\right)}    \end{align}
\end{cor}
In case $V=M(n,\K{R})$ of all square matrices we provide an explicit formula for the expected condition number, that is the claim of our Theorem \ref{gaussian} above.

We give an explicit answer also in the case $V=Sym(n,\K{R})$ of symmetric matrices. In this case the probability space $(Sym(n,\K{R}),\C{N}_{Sym(n,\K{R})})$ is usually referred to as \emph{Gaussian Orthogonal Ensemble} (GOE).
\begin{cor}\label{GOE}
If $A_0,\dots,A_d\in Sym(n,\K{R})$ are independent $GOE(n)$-matrices and $n$ is even, then
\begin{align}\label{eq:GOE}
      \underset{A_0,\dots,A_d\sim i.i.d.\, GOE(n)}{\mean}\, \mu(A) &=  \sqrt{2}n \frac{\Gamma\left(\frac{(d+1)n(n+1)}{4}\right)}{\Gamma\left(\frac{(d+1)n(n+1)-2}{4}\right)}\frac{\G{n+1}{2}}{\G{n+2}{2}}\\
      &= \sqrt{(d+1)n^3}\left(1+\C{O}\left(\frac{1}{\sqrt{n}}\right)\right),\ n\rightarrow +\infty
    \end{align}
If $n$ is odd the explicit formula is more complicated and is given in the proof of the corollary. However the above asymptotic formula is valid for both even and odd $n$.
\end{cor}

\section{Preliminaries}\label{preliminaries}
Below we state few classical results in semialgebraic geometry that we will use, the proofs can be found in \cite{BCR:98,Coste}.

Given a semialgebraic set $S\subset \K{R}^N$ of dimension $k\leq N$ we fix a \emph{semialgebraic stratification} of $S$, i.e., a partition of $S$ into finitely many semialgebraic subsets (called \emph{strata}) such that each stratum is a smooth submanifold of $\K{R}^N$ and  the boundary of any stratum of dimension $i\leq N$ is a union of some strata of dimension less than $i$. We denote by $S_{\txt{top}}$ the union of all $k$-dimensional strata of $S$ and by  $S_{\txt{low}} = S\setminus S_{\txt{top}}$ the union of the strata of dimension less than $k$. The sets $S_{\txt{top}}, S_{\txt{low}}\subset \K{R}^N$ are semialgebraic and $S_{\txt{top}}$ is a smooth $k$-dimensional submanifold of $\K{R}^N$.

One of the central results about semialgebraic mappings is Hardt's theorem.
\begin{thm}[Hardt's semialgebraic triviality]\label{Hardt}
Let $S\subset \K{R}^N$ be a semialgebraic set and let $f: S\rightarrow \K{R}^M$ be a continuous semialgebraic mapping. Then there exists a finite partition of $\K{R}^M$ into semialgebraic sets $C_1,\dots,C_r\subset \K{R}^M$ such that $f$ is semialgebraically trivial over each $C_i$, i.e., there are a semialgebraic set $F_i$ and a semialgebraic homeomorphism $h_i: f^{-1}(C_i) \rightarrow C_i\times F_i$ such that the composition of $h_i$ with the projection $C_i\times F_i\rightarrow C_i$ equals $f|_{f^{-1}(C_i)}$.
\end{thm}
The following corollary of Hardt's theorem is frequently used to estimate dimension of semialgebraic sets.
\begin{cor}\label{cor:Hardt} Let $f: S\rightarrow \K{R}^M$ be as above. Then the set $\{x\in \K{R}^N: \txt{dim}(f^{-1}(x)) = d\}$ is semialgebraic and has dimension not greater than $\txt{dim}(S)-d$. 
\end{cor}
\section{Proof of main results}\label{main}
In this section we prove our main results, Theorems \ref{general} and \ref{pevp general}. 
Let us first fix some notations that are used in the rest of the paper: for a non-degenerate subset $\C{S}\subset \K{R}^m\times S^1$ by $\Sigma_1, \Sigma_2 \subset \C{S}_{\txt{top}}$ we denote the semialgebraic sets of critical points of $p_1: \C{S}_{\txt{top}} \rightarrow \K{R}^m$ and $p_2: \C{S}_{\txt{top}} \rightarrow S^1$ respectively, the corresponding semialgebraic sets of critical values are denoted by $\sigma_1=p_1(\Sigma_1)\subset \K{R}^m$ and $\sigma_2=p_2(\Sigma_2)\subset S^1$. 
\subsection{Proof of Theorem \ref{general}}
In this subsection $\C{S}$ denotes a non-degenerate semialgebraic subset of $\K{R}^m\times S^1$. 
For the proof of Theorem \ref{general} we need few technical lemmas which we state and prove below.
\begin{lemma}\label{lemma1}
The semialgebraic sets $\C{S}\subset \K{R}^m\times S^1$ and $p_1(\C{S})\subset \K{R}^m$ are of dimension $m$. 
\end{lemma}
\begin{proof}
 Since $\C{S}$ is non-degenerate, for every $x\in S^1$ the fiber $p_2^{-1}(x)$ is $(m-1)$-dimensional. From Theorem \ref{Hardt} it follows that for some $x\in S^1$ we have $\txt{dim}(\C{S}) = \txt{dim}(p_2^{-1}(x))+\txt{dim}(S^1) = (m-1)+1 = m$. 

The map $p_1: \C{S}_{\txt{top}} \rightarrow \K{R}^m$ has a regular point $(a,x)\in \C{S}_{\txt{top}}\setminus \Sigma_1$ since $\C{S}$ is $m$-dimensional and the set $\Sigma_1$ of critical points of $p_1$ is at most $(m-1)$-dimensional. The image $p_1(U)$ of a small open neighbourhood $U\subset \C{S}_{\txt{top}}\setminus \Sigma_1$ of $(a,x)\in U$ is open in $\K{R}^m$ and hence $\txt{dim}(p_1(\C{S})) =m$.
\end{proof}

\begin{lemma}\label{lemma2}
There exists an open semialgebraic subset $M\subset \C{S}_{\txt{top}}$ such that $p_1(M)$ is open in $\K{R}^m$, $M= p_1^{-1}(p_1(M))$, the restriction $p_1: M \rightarrow p_1(M)$ is a submersion and $\txt{dim}\,(\C{S}\setminus M)\leq m-1$. 
\end{lemma}
\begin{proof}
Define $M:=p_1^{-1}(\K{R}^m\setminus N) = \C{S}_{\txt{top}}\setminus p_1^{-1}(N)$, where $N:=\overline{p_1(\C{S}_{\txt{low}} \cup \Sigma_1)}$ and the bar stands for the euclidean closure of a set. Note that $M$ is an open subset of $\C{S}_{\txt{top}}$ and $M= p_1^{-1}(p_1(M))$. Moreover $M$ consists of regular points of the projection $p_1:\C{S}_{\txt{top}}\rightarrow \K{R}^m$, which implies that $p_1(M)$ is an open subset of $\K{R}^m$ and $p_1:M\rightarrow p_1(M)$ is a submersion of smooth manifolds. Indeed, for $a\in p_1(M)$ and  $(a,x)\in M$ the image $p_1(U)$ of a small open neighborhood $U\subset M$ of $(a,x)\in U$ is open in $\K{R}^m$ and $a\in p_1(U)$.

We now prove that $\C{S}\setminus M = p_1^{-1}(N)$ is at most $(m-1)$-dimensional. Since $\C{S}$ is non-degenerate there exists a semialgebraic set $B\subset \K{R}^m$ with $\txt{dim}(B)\leq m-2$ such that $p_1^{-1}(a)$ is finite for $a\notin B$. We decompose the semialgebraic set $N = (N\cap B) \cup (N\setminus B)$. From Theorem \ref{Hardt} it follows that there exists some $a\in N\cap B$ such that  ${\txt{dim}(p_1^{-1}(N\cap B))} \leq \txt{dim}(p_1^{-1}(a)) +\txt{dim}(N\cap B)\leq 1 + (m-2) = m-1$. For $a\in N\setminus B$ the fiber $p_1^{-1}(a)$ is discrete, which together with the non-degeneracy of $\C{S}$ and Theoren \ref{Hardt} implies $\txt{dim}(p_1^{-1}(N\setminus B)) \leq \txt{dim}(p_1^{-1}(a))+\txt{dim}(N\setminus B) \leq  \txt{dim}(\C{S}_{\txt{low}} \cup \Sigma_1) \leq m-1$. Therefore, $\txt{dim}(\C{S}\setminus M) = \txt{dim}(p_1^{-1}(N)) = \txt{dim}(p_1^{-1}(N\cap B)\cup p_1^{-1}(N\setminus B))\leq m-1$.
\end{proof}
\begin{lemma}\label{lemma3}
There exists an open semialgebraic subset $R\subset \C{S}_{\txt{top}}$ such that $S^1\setminus p_2(R)$ is finite, $p_2: R\rightarrow p_2(R)$ is a submersion, $\txt{dim}(\C{S}\setminus R)\leq m-1$ and $\txt{dim} (p_2^{-1}(x)\setminus R) \leq m-2$ for $x\in p_2(R)$.
\end{lemma}
\begin{proof}
Since $\C{S}$ is non-degenerate every fiber $p_2^{-1}(x), x\in S^1$ is $(m-1)$-dimensional. 

Note that the set $S^1\setminus p_2(\C{S}_{\txt{top}})$ is semialgebraic and zero-dimensional, thus finite. Indeed, if it was one-dimensional Theorem \ref{Hardt} together with $\txt{dim}(p_2^{-1}(x))=m-1, x\in S^1$ would imply that $p_2^{-1}(S^1\setminus p_2(\C{S}_{\txt{top}}))\subset \C{S}\setminus \C{S}_{\txt{top}}$ is $m$-dimensional which would contradict to $\txt{dim}(\C{S}\setminus \C{S}_{\txt{top}})\leq m-1$. 

The semialgebraic set $\sigma_2=p_2(\Sigma_2)\subset S^1$ of critical values of $p_2: \C{S}_{\txt{top}} \rightarrow S^1$ has measure zero by Sard's theorem. Hence $\sigma_2\subset S^1$ consists of a finite number of points. 

Applying Corollary \ref{cor:Hardt} to the map $p_2: \C{S}_{\txt{low}} \rightarrow S^1$ we have that $C:=\{x\in S^1: \txt{dim}(p_2^{-1}(x)\cap \C{S}_{\txt{low}}) = m-1\}$ is a semialgebraic subset of $S^1$ and $\txt{dim}(C)\leq \txt{dim}(\C{S}_{\txt{low}})-(m-1)\leq 0$. Thus $C$ is a (possibly empty) finite set.

Set now $R:=\C{S}_{\txt{top}}\setminus p_2^{-1}(\sigma_2\cup C)$. Note that $R$ is an open semialgebraic subset of $\C{S}_{\txt{top}}$ and $S^1\setminus p_2(R) = \sigma_2\cup C\cup(S^1\setminus p_2(\C{S}_{\txt{top}}))$ is finite by the above arguments. Since $R$ consists of regular points of $p_2:\C{S}_{\txt{top}}\rightarrow S^1$ the map $p_2: R\rightarrow p_2(R)$ is a submersion. Since $\txt{dim}(\C{S}_{\txt{low}})\leq m-1$ and $p_2^{-1}(\sigma_2\cup C)$ is a finite collection of $(m-1)$-dimensional fibers we have that $\txt{dim}(\C{S}\setminus R = \C{S}_{\txt{low}} \cup p_2^{-1}(\sigma_2\cup C))\leq m-1$. Finally, $\txt{dim}(p_2^{-1}(x)\setminus R= p_2^{-1}(x) \cap\C{S}_{\txt{low}})\leq m-2$ for $ x\in p_2(R)$ because $p_2(R)\cap C=\varnothing$.
\end{proof}

\begin{lemma}\label{coarea}
 For any measurable function $f: \C{S} \rightarrow [0,+\infty)$ we have
 \begin{align}
   \int\limits_{a\in \K{R}^m} \sum\limits_{x\in S^1 :\, (a,x)\in \C{S}} f(a,x)\, da = \int\limits_{x\in S^1} \int\limits_{a\in p_2^{-1}(x)} \frac{NJ_{(a,x)} p_1}{NJ_{(a,x)} p_2} f(a,x)\, da\,dx
 \end{align}
\end{lemma}
\begin{proof}
Let $M\subset \C{S}_{\txt{top}}$ be as in Lemma \ref{lemma2}. The smooth coarea formula \cite[(A-2)]{Howard} applied to the measurable function $f: M \rightarrow [0,+\infty)$ and to the submersion $p_1: M\rightarrow p_1(M)$ reads
\begin{align}\label{coarea1}
  \int\limits_{(a,x)\in M} NJ_{(a,x)} p_1\, f(a,x)\, d(a,x) = \int\limits_{a\in p_1(M)} \sum\limits_{x\in S^1:\, (a,x)\in \C{S}} f(a,x)\, da,
\end{align}
where we used that $M=p_1^{-1}(p_1(M))$ (Lemma \ref{lemma2}) to be able to sum over the whole fiber $p_1^{-1}(a)=\{(a,x)\in \C{S}\},\, a\in p_1(M)$.
By Lemma \ref{lemma2} we have $\txt{dim}(\C{S}\setminus M)\leq m-1$ and hence $\txt{dim}(p_1(\C{S})\setminus p_1(M)=p_1(\C{S}\setminus M))\leq \txt{dim}(\C{S}\setminus M)\leq m-1$. Thus we extend the integrations in \eqref{coarea1} over $\C{S}$ and $p_1(\C{S})$ respectively without changing the result. Moreover the integration over $p_1(\C{S})$ can be further extended to the whole space $\K{R}^m$ since for a point $a\in \K{R}^m\setminus p_1(\C{S})$ the summation $\sum_{x\in S^1: (a,x)\in \C{S}} f(a,x)$ is performed over the empty set $p_1^{-1}(a)$ in which case the sum is conventionally set to $0$. All together the above arguments imply
\begin{align}\label{coarea2}
  \int\limits_{(a,x)\in \C{S}} NJ_{(a,x)} p_1\, f(a,x)\, d(a,x) = \int\limits_{a\in \K{R}^m} \sum\limits_{x\in S^1:\, (a,x)\in \C{S}} f(a,x)\, da,
\end{align}

Let $R\subset \C{S}_{\txt{top}}$ be as in Lemma \ref{lemma3}. Applying the smooth coarea formula \cite[(A-2)]{Howard} to the measurable function $\frac{NJ p_1}{NJ p_2} f: R \rightarrow [0,+\infty)$ and to the submersion $p_2: R\rightarrow p_2(R)$ we obtain
\begin{align}\label{coarea3}
\int\limits_{(a,x)\in R} NJ_{(a,x)} p_1\, f(a,x)\, d(a,x) = \int\limits_{x\in p_2(R)} \int\limits_{a\in p_2^{-1}(x)\cap R} \frac{NJ_{(a,x)} p_1}{NJ_{(a,x)} p_2} f(a,x)\, da\, dx
\end{align}
By Lemma \ref{lemma3} $\txt{dim}(\C{S}\setminus R)\leq m-1$, $S^1\setminus p_2(R)$ is finite, and $\txt{dim}(p_2^{-1}(x)\setminus R)\leq m-2$ for $x\in p_2(R)$. Thus the integrations in \eqref{coarea3} can be extended over $\C{S}, S^1$ and $p_2^{-1}(x)$ respectively leading to
\begin{align}\label{coarea4}
\int\limits_{(a,x)\in \C{S}} NJ_{(a,x)} p_1\, f(a,x)\, d(a,x) = \int\limits_{x\in S^1} \int\limits_{a\in p_2^{-1}(x)} \frac{NJ_{(a,x)} p_1}{NJ_{(a,x)} p_2} f(a,x)\, da\, dx
\end{align}
Combining \eqref{coarea4} with \eqref{coarea2} we finish the proof.
\end{proof}
Now comes the proof of Theorem \ref{general}.
\begin{proofof}{Theorem \ref{general}}
The following identity is the key point of the proof:
\begin{align}\label{identity}
 \mu(a,x) = \Vert a \Vert\, \frac{NJ_{(a,x)}p_2}{NJ_{(a,x)}p_1},\quad (a,x)\in M\cap R\subset \C{S}_{\txt{top}} 
\end{align}
where $M\subset \C{S}_{\txt{top}}$ and $R\subset \C{S}_{\txt{top}}$ are as in Lemma \ref{lemma2} and Lemma \ref{lemma3} respectively and $\mu(a,x)$, the local condition number of $(a,x)\in \C{S}$, is defined in Definition \ref{condition}.
The proof of the identity comes after we derive the statement of Theorem \ref{general}. 

Applying Lemma \ref{coarea} to the measurable function $f(a,x) = \mu(a,x) e^{-\Vert a\Vert^2/2}/\sqrt{2\pi}^m, (a,x)\in \C{S}$, and using \eqref{identity} we obtain:
\begin{align}
  \mean_{a\sim N(0,1)} \left( \sum\limits_{x\in S^1:\,(a,x)\in \C{S}} \mu(a,x) \right) &= \frac{1}{\sqrt{2\pi}^m}\int\limits_{a\in \K{R}^m} \left(\sum\limits_{x\in S^1:\, (a,x)\in \C{S}} \mu(a,x)\right) e^{-\frac{\Vert a\Vert^2}{2}} da \\
&=\frac{1}{\sqrt{2\pi}^m} \int\limits_{x\in S^1} \int\limits_{a\in p_2^{-1}(x)}\Vert a\Vert e^{-\frac{\Vert a\Vert^2}{2}} da\,dx = (*),
\end{align}
which gives the claimed formula \eqref{general1}.
If $\C{S}$ is scale-invariant with respect to $a\in \K{R}^m$ by Lemma \ref{A:lemma1} we have 
\begin{align}
  (*) = \frac{\Gamma\left(\frac{m}{2}\right)}{2\sqrt{\pi}^m} \int\limits_{x\in S^1} |p_2^{-1}(x)\cap S^{m-1}|\, dx.
\end{align}

Now we turn to the proof of \eqref{identity}.

For $(a,x)\in M\cap R\subset \C{S}_{\txt{top}}$ let $(\dot{a}_0,\dot{x}_0),(\dot{a}_1,0),\dots,(\dot{a}_{m-1},0)$ be an orthonormal basis of $T_{(a,x)} R$ with $(\dot{a}_j,0)\in \ker D_{(a,x)} p_2,\, j=1,\dots,m-1$. Note that $\dot{a}_0\in \K{R}^m, \dot{x}_0\in T_x S^1$ are non-zero since $p_1: M\rightarrow p_1(M), p_2: R\rightarrow p_2(R)$ are submersions and $\dot{a}_0\in \K{R}^m$ is orthogonal to $\dot{a}_j\in \K{R}^m,\, j=1,\dots,m-1$. We compute the normal Jacobians $NJ_{(a,x)} p_1$ and $NJ_{(a,x)} p_2$ using the following orthonormal bases:
\begin{align}
\left\{(\dot{a}_0,\dot{x}_0), (\dot{a}_1,0),\dots,(\dot{a}_{m-1},0)\right\} &\subset T_{(a,x)} \C{S}_{\txt{top}}\\
\left\{\frac{\dot{a}_0}{\Vert \dot{a}_0\Vert}, \dot{a}_1,\dots,\dot{a}_{m-1}\right\} &\subset T_a\K{R}^m\\
\left\{\frac{\dot{x}_0}{\Vert \dot{x}_0\Vert}\right\} &\subset T_x S^1
\end{align}
It is straightforward to see that $NJ_{(a,x)} p_1 = \Vert \dot{a}_0\Vert$ and $NJ_{(a,x)} p_2 = \Vert \dot{x}_0\Vert$ and hence
\begin{align}\label{ratio}
\frac{NJ_{(a,x)} p_2}{NJ_{(a,x)} p_1} = \frac{\Vert \dot{x}_0\Vert}{\Vert \dot{a}_0\Vert}
\end{align}
Since $D_{a}( p_2\circ p_1^{-1})(\dot{a}_j) = D_{(a,x)} p_2\circ D_a p^{-1}_1(\dot{a}_j) = 0$ for $j=1,\dots,m-1$ and since $D_{a}( p_2\circ p_1^{-1})(\dot{a}_0) = D_{(a,x)} p_2\circ D_a p^{-1}_1(\dot{a}_0) = \dot{x}_0$ we obtain
\begin{align}
\mu(a,x) = \Vert a \Vert\,\sup\limits_{\dot a \in \K{R}^m\setminus \{0\}} \frac{\Vert D_a( p_2\circ p^{-1}_1)(\dot a)\Vert}{\Vert\dot a\Vert} = \Vert a\Vert \frac{\Vert \dot{x}_0\Vert}{\Vert \dot{a}_0\Vert}
\end{align}
This together with \eqref{ratio} implies the claimed identity \eqref{identity}.
\end{proofof}

\subsection{Proof of Theorem \ref{pevp general}}
For a $k$-dimensional vector subspace $V\subset M(n,\K{R})$ and for a basis $f=(f_0(\alpha,\beta),\dots,f_d(\alpha,\beta))$ of the space $P_{d,2}$ of binary forms of degree $d\geq 1$ let us define the algebraic variety
\begin{align*}
\C{S}(V,f):=\{(A,x)\in V^{d+1}\times S^1 : \txt{det}(A_0\, f_0(\alpha,\beta) +\dots+A_d\, f_d(\alpha,\beta)) = 0\}
\end{align*}
Theorem \ref{pevp general} follows from the following more general result.
\begin{thm}\label{general basis}
If $\Sigma_V\subset V$ is of codimension one and $f$ is any basis of $P_{d,2}$, then $\C{S}(V,f)$ is non-degenerate and 
\begin{align*}
      \underset{A_0,\dots,A_d\sim i.i.d.\, \mathcal{N}_{V}}{\mean} \left( \sum\limits_{[x]\in \RP:\, (A,x)\in \C{S}(V,f)} \mu(A,x) \right) =  \sqrt{\pi} \frac{\Gamma\left(\frac{(d+1)k}{2}\right)}{\Gamma\left(\frac{(d+1)k-1}{2}\right)}\frac{|\Sigma_V\cap S^{k-1}|}{|S^{k-2}|},
\end{align*}
\end{thm}
\begin{proof}
Observe first that for any $x=(\alpha,\beta)\in S^1$ the vector $f(x)=(f_0(\alpha,\beta),\dots,f_d(\alpha,\beta))$ is non-zero. For any such fixed $x$, let $g=(g_{ij})\in O(d+1)$ be an orthogonal matrix that sends $f(x)$ to $(c,0,\dots,0)\in \K{R}^{d+1}\setminus \{0\}$, where $c\neq 0$ is some constant, i.e., 
\begin{align*}
\sum\limits_{j=0}^d g_{ij} f_j(\alpha,\beta) =
\begin{cases}
c, i=0,\\
0, i=1,\dots, d
\end{cases}
\end{align*}
It is easy to verify that the linear change of coordinates $A_j = \sum_{i=0}^d g_{ij} \tilde{A}_i,\, j=0,\dots,d$ is an isometry of the product space $(V,(\cdot,\cdot))^{d+1}$ and 
\begin{align*}
\sum\limits_{j=0}^d f_j(\alpha,\beta) A_j &= \sum\limits_{j=0}^d f_j(\alpha,\beta) \left(\sum\limits_{i=0}^d g_{ij} \tilde{A}_i\right) = 
 \sum\limits_{i=0}^d\left(\sum\limits_{j=0}^d g_{ij} f_j(\alpha,\beta)\right) \tilde{A}_i = c\, \tilde{A}_0
\end{align*}
Therefore, for $x=(\alpha,\beta)\in S^1$ there is a global isometry $\C{I}_x : (V,(\cdot,\cdot))^{d+1} \rightarrow (V,(\cdot,\cdot))^{d+1}$ that sends the fiber $p_2^{-1}(x) = \{A\in V^{d+1}: \txt{det}(A_0\,f_0(\alpha,\beta)+\dots+ A_d\,f_d(\alpha,\beta))=0\}$ to $\{\tilde{A}\in V^{d+1}: \txt{det} (\tilde{A}_0) = 0\} = \Sigma_V \times V^d$. In particular, under the assumption $\txt{dim}(\Sigma_V) = k-1$ we have that $p_2^{-1}(x)$ is of codimension one in $V^{d+1}$ and hence condition $(1)$ in Definition \ref{def:S} is satisfied. 

Since both $f_0(\alpha,\beta),\dots, f_d(\alpha,\beta)$ and $\alpha^0\beta^d,\dots,\alpha^d\beta^0$ are bases of $P_{d,2}$ for some ${h=(h_{ij})\in GL(d+1)}$ we have $\alpha^i \beta^{d-i} = \sum_{j=0}^d h_{ij} f_j(\alpha,\beta),\, i=0,\dots,d$. Let us define $B = \{A\in V^{d+1}: A_j = \sum\limits_{i=0}^d h_{ij}\tilde A_i,\, j=0,\dots,d,\ \txt{det}(\tilde A_0) = \txt{det}(\tilde A_d) = 0\}$. Since $\txt{dim}(\Sigma_V) = k-1$ and since $h$ is a non-degenerate transformation the algebraic subset $B\subset V^{d+1}$ has codimension $2$. For $A\notin B$ the matrix  
\begin{align*}
\sum\limits_{j=0}^d f_j(\alpha,\beta) A_j = \sum\limits_{j=0}^d f_j(\alpha,\beta) \left(\sum\limits_{i=0}^d h_{ij} \tilde A_i\right) = \sum\limits_{i=0}^d\left(\sum\limits_{j=0}^d h_{ij} f_j(\alpha,\beta)\right) \tilde A_i = \sum\limits_{i=0}^d \alpha^i\beta^{d-i} \tilde{A}_i
\end{align*}
is non-degenerate at $(\alpha:\beta)=(0,1)$ (at $(\alpha ,\beta)=(1,0)$) if $\txt{det}(\tilde{A}_0)\neq 0$ (if $\txt{det}(\tilde{A}_d)\neq 0$, respectively) and hence the binary form $\txt{det}(A_0 f_0(\alpha,\beta) + \dots+ A_d f_d(\alpha,\beta))$ is non-zero. Consequently, the fiber $p_1^{-1}(A)= \{x\in S^1: \txt{det}(A_0 f_0(\alpha,\beta)+ \dots + A_d f_d(\alpha,\beta))= 0\}$ is finite for any $A\notin B$ and condition $(2^{\prime})$ in Definition \ref{def:S} is satisfied. Applying Corollary \ref{cor:projective} to  $\C{S}(V,f)\subset \K{R}^{(d+1)k} \times S^1,\, \K{R}^{(d+1)k} \simeq V^{d+1}$ we obtain
\begin{align}
  \hspace{-4pt}\underset{A_0,\dots, A_d\sim i.i.d.\, \mathcal{N}_{V}}{\mean} \hspace{-2pt}\left( \sum\limits_{[x]\in \RP:\, (A,x)\in \C{S}(V,f)} \mu(A,x) \right)\hspace{-2pt} = \hspace{-1pt}\frac{1}{\sqrt{2\pi}^{(d+1)k}}\hspace{-1pt} \int\limits_{[x]\in \RP} \int\limits_{A\in p_2^{-1}(x)}\Vert A\Vert e^{-\frac{\Vert A\Vert^2}{2}} dA\,dx,
\end{align}
where $\Vert A\Vert^2 = \Vert A_0\Vert^2+\dots+\Vert A_d\Vert^2$. 
Since each fiber $p_2^{-1}(x), x\in S^1$ is an algebraic subset of $\K{R}^{(d+1)k}$ of codimension one we have by Lemma \ref{A:lemma1}
\begin{align}
\int\limits_{A\in p_2^{-1}(x)}\Vert A\Vert e^{-\frac{\Vert A\Vert^2}{2}} dA = \sqrt{2}\, \frac{\G{(d+1)k}{2}}{\G{(d+1)k-1}{2}} \int\limits_{A\in p_2^{-1}(x)} e^{-\frac{\Vert A\Vert^2}{2}}	 dA
\end{align}
Performing the isometric change of coordinates $\C{I}_x: (V,(\cdot,\cdot))^{d+1} \rightarrow (V,(\cdot,\cdot))^{d+1}$ that was constructed above we write the last integral as follows:
\begin{align}
\int\limits_{A\in p_2^{-1}(x)} e^{-\frac{\Vert A\Vert^2}{2}} dA &= \int\limits_{\{\tilde A \in V^{d+1} :\, \txt{det}(\tilde{A}_0) = 0\}} e^{-\frac{\Vert \tilde A_0\Vert^2}{2}} e^{-\frac{\Vert \tilde A_1\Vert^2}{2}}\dots e^{-\frac{\Vert \tilde A_d\Vert^2}{2}} d\tilde{A}_0 d\tilde{A}_1\dots d\tilde{A}_d\\
&= \sqrt{2\pi}^{\,dk} \int\limits_{\tilde{A}_0\in \Sigma_V} e^{-\frac{\Vert \tilde{A}_0\Vert^2}{2}} d\tilde{A}_0 = \sqrt{2\pi}^{\,dk} \sqrt{2}^{\,k-3} \G{k-1}{2} |\Sigma_V\cap S^{k-1}|,
\end{align}
where in the last step Lemma \ref{A:lemma1} has been used. Collecting everything together we write
\begin{align}
 \underset{A_0,\dots, A_d\sim i.i.d.\, \mathcal{N}_{V}}{\mean} \left( \sum\limits_{[x]\in \RP:\, (A,x)\in \C{S}(V,f)} \mu(A,x) \right) &= \sqrt{\pi}\frac{\G{(d+1)k}{2}}{\G{(d+1)k-1}{2}} \frac{\G{k-1}{2}}{2\sqrt{\pi}^{k-1}} |\Sigma_V\cap S^{k-1}|\\
&= \sqrt{\pi} \frac{\G{(d+1)k}{2}}{\G{(d+1)k-1}{2}} \frac{|\Sigma_V\cap S^{k-1}|}{|S^{k-2}|},
\end{align} 
since $|S^{k-2}| = 2\sqrt{\pi}^{k-1}/\G{k-1}{2}$. This completes the proof.
\end{proof}

\begin{proofof}{Theorem \ref{pevp general}}
Taking $f_i(\alpha,\beta) = \alpha^i\beta^{d-i},\, i=0,\dots,d$ in Theorem \ref{general basis} we obtain the claim of Theorem \ref{pevp general}.
\end{proofof}

\section{Applications of main results}\label{applications}
In this section we derive Theorem \ref{gaussian} and Corollaries \ref{bound}, \ref{GOE}. 
\begin{proofof}{Corollary \ref{bound}}
Applying Poincar\'e's formula \cite[(3-5)]{Howard} to the projective hypersurface $\txt{P}\Sigma_V\subset \txt{P}V\simeq \K{R}\txt{P}^{k-1}$ and a projective line $\ell \in \mathbb{G}(1,k-1)$ we obtain
\begin{align}
\frac{|\Sigma_V\cap S^{k-1}|}{|S^{k-2}|} = \frac{|\txt{P}\Sigma_V|}{|\K{R}\txt{P}^{k-2}|} = \underset{\ell\in \mathbb{G}(1,k-1)}{\mean}\#(\txt{P}\Sigma_V\cap \ell) \leq \txt{deg}(\txt{P}\Sigma_V) = n
\end{align}
which together with \eqref{eq:pevp general} implies the claimed bound \eqref{eq:bound}. 
\end{proofof}
In case of any particular space $V\subset M(n,\K{R})$ satisfying $\txt{dim}(\Sigma_V) = k-1 = \txt{dim}(V)-1$ by Theorem \ref{pevp general} explicit computation of the expected condition number for polynomial eigenvalues amounts to computing the volume of the hypersurface $\Sigma_V\cap S^{k-1}$. In cases $V=M(n,\K{R})$ and $V=Sym(n,\K{R})$ formulas for the volume of $\Sigma_V\cap S^{k-1}$ were found in \cite{EKS94} and \cite{LerLun16} respectively. 
\begin{proofof}{Theorem \ref{gaussian}}
Formula from \cite{EKS94} reads
\begin{align}
\frac{|\Sigma_{M(n,\K{R})}\cap S^{n^2-1}|}{|S^{n^2-2}|} = \sqrt{\pi}\frac{\G{n+1}{2}}{\G{n}{2}}
\end{align}
Plugging it in \eqref{eq:pevp general} for $V=M(n,\K{R}), k=\txt{dim}(V)=n^2$ leads to
\begin{align}
      \underset{A_0,\dots,A_d\sim i.i.d.\, \mathcal{N}_{M(n,\K{R})}}{\mean}\, \mu(A) &= \pi \frac{\Gamma\left(\frac{(d+1)n^2}{2}\right)}{\Gamma\left(\frac{(d+1)n^2-1}{2}\right)} \frac{\G{n+1}{2}}{\G{n}{2}} \\
      &= \frac{\pi}{2} \sqrt{(d+1)n^3}\left(1 + \C{O}\left(\frac{1}{n}\right)\right),\ n\rightarrow +\infty,
\end{align}
where the asymptotic is obtained using formula $(1)$ from \cite{TriErd51}.
\end{proofof}
\begin{proofof}{Corollary \ref{GOE}}
In \cite{LerLun16} it was proved that 
\begin{align}\label{eq:volume even}
\frac{|\Sigma_{Sym(n,\K{R})}\cap S^{\frac{n(n+1)}{2}-1}|}{|S^{\frac{n(n+1)}{2}-2}|} = 
\sqrt{\frac{2}{\pi}}n \frac{\G{n+1}{2}}{\G{n+2}{2}}
\end{align}
for even $n$ and 
\begin{align}\label{eq:volume odd}
\frac{|\Sigma_{Sym(n,\K{R})}\cap S^{\frac{n(n+1)}{2}-1}|}{|S^{\frac{n(n+1)}{2}-2}|}=\frac{(-1)^m\sqrt{\pi}n!}{2^n m!\G{n+2}{2}} \left(1-\frac{4\sqrt{2}}{\sqrt{\pi}}\sum\limits_{i=0}^{m-1} (-1)^i \frac{\G{2i+3}{2}}{i!}\right)
\end{align}
for odd $n=2m+1$. Plugging \eqref{eq:volume even} and \eqref{eq:volume odd} in \eqref{eq:pevp general} for $V=Sym(n,\K{R}), k=\frac{n(n+1)}{2}$ leads to explicit formulas for the expected condition number (see \eqref{eq:GOE} in case of even $n$). In \mbox{\cite[Remark 3]{LerLun16}} it was shown that 
\begin{align}
\frac{|\Sigma_{Sym(n,\K{R})}\cap S^{\frac{n(n+1)}{2}-1}|}{|S^{\frac{n(n+1)}{2}-2}|}=\frac{2\sqrt{n}}{\sqrt{\pi}}\left(1+\C{O}\left(\frac{1}{\sqrt{n}}\right)\right),\ n\rightarrow +\infty
\end{align}
regardless parity of $n$. This leads to the asymptotic
\begin{align}
\underset{A_0,\dots,A_d\sim i.i.d.\, GOE(n)}{\mean}\, \mu(A) &=  \sqrt{\pi} \frac{\Gamma\left(\frac{(d+1)n(n+1)}{4}\right)}{\Gamma\left(\frac{(d+1)n(n+1)-2}{4}\right)}\frac{|\Sigma_{Sym(n,\K{R})}\cap S^{\frac{n(n+1)}{2}-1}|}{|S^{\frac{n(n+1)}{2}-2}|}\\
           &=\sqrt{(d+1)n^3}\left(1+\C{O}\left(\frac{1}{\sqrt{n}}\right)\right),\ n\rightarrow +\infty,
\end{align} 
where we again used formula $(1)$ from \cite{TriErd51} for the asymptotic of the ratio of two Gamma functions. 
\end{proofof}

\section*{Appendix}
In the following proposition we show that the conditions $(2)$ and $(2^{\prime})$ in Definition \ref{def:S} of a non-degenerate semialgebraic set $\C{S}\subset \K{R}^m\times S^1$ are equivalent.
\begin{propo}\label{equivalence}
Let $\C{S}\subset \K{R}^m\times S^1$ be a semialgebraic subset of dimension $m$. Then

$(2)$ the semialgebraic set $\Sigma_1\subset \C{S}_{\txt{top}}$ of critical points of the first projection $p_1:\C{S}_{\txt{top}} \rightarrow \K{R}^m$ is at most $(m-1)$-dimensional if and only if 

$(2^{\prime})$
there exists a semialgebraic subset $B\subset \K{R}^m$ of dimension at most $m-2$ such that for any $a\notin B$ the fiber $p_1^{-1}(a)$ is finite.
\end{propo}
\begin{proof}
$(2) \Rightarrow (2^{\prime})$ By Sard's theorem the semialgebraic set $\sigma_1=p_1(\Sigma_1)\subset \K{R}^m$ of critical values of $p_1:\C{S}_{\txt{top}} \rightarrow \K{R}^m$ is of dimension $\leq m-1$. The set $p_1^{-1}(\sigma_1)\subset \C{S}$ of critical fibers is also of dimension $\leq m-1$. Indeed, if it was $m$-dimensional there would exist a nonempty open set $U\subset p_1^{-1}(\sigma_1)\setminus (\Sigma_1 \cup \C{S}_{\txt{low}})$ of regular points of $p_1$. The image $p_1(U)\subset \sigma_1$ of $U$ is open in $\K{R}^m$ which contradicts to $\txt{dim}(\sigma_1)\leq m-1$.

For the map $p_1: p_1^{-1}(\sigma_1) \rightarrow \sigma_1$ define 
$B_1:=\{a\in \sigma_1: \txt{dim}(p_1^{-1}(a)) =1\}$, the semialgebraic set of points in $\sigma_1$ for which the fiber $p_1^{-1}(a)$ is infinite. Since $\txt{dim}(p_1^{-1}(\sigma_1))\leq m-1$ Corollary \ref{Hardt} implies that $\txt{dim}(B_1)\leq m-2$.

Similarly, for the map $p_1: \C{S}_{\txt{low}} \rightarrow p_1(\C{S}_{\txt{low}})$ let us define $B_2:=\{a\in p_1(\C{S}_{\txt{low}}) : \txt{dim}(p_1^{-1}(a)\cap \C{S}_{\txt{low}})=1\}$, the semialgebraic set of points in $p_1(\C{S}_{\txt{low}})$ for which the fiber $p_1^{-1}(a)\cap \C{S}_{\txt{low}}$ is infinite. Since $\txt{dim}(\C{S}_{\txt{low}})\leq m-1$ Corollary \ref{Hardt} gives $\txt{dim}(B_2) \leq m-2$.

Take now any $a\notin B_1\cup B_2$. If $a\in \sigma_1$ the fiber $p_1^{-1}(a)$ is finite since $a\notin B_1$. If $a \notin \sigma_1$ it's a regular point of the map $p_1: \C{S}_{\txt{top}} \rightarrow \K{R}^m$ between two $m$-dimensional manifolds. Therefore the semialgebraic set $p_1^{-1}(a)\cap \C{S}_{\txt{top}}$ is zero-dimensional manifold and hence it's finite. The set $p_1^{-1}(a)\cap \C{S}_{\txt{low}}$ is finite because $a\notin B_2$. Consequently, the fiber $p_1^{-1}(a) =( p_1^{-1}(a)\cap \C{S}_{\txt{top}}) \cup (p_1^{-1}(a)\cap \C{S}_{\txt{low}})$ is finite for any point $a\notin B$ out of the at most $(m-2)$-dimensional semialgebraic subset $B:=B_1\cup B_2\subset \K{R}^m$.

$(2)\Leftarrow (2^{\prime})$ Recall that $\txt{dim}(\sigma_1)\leq m-1$ and let us consider the map $p_1: \Sigma_1 \rightarrow \sigma_1$. If $\Sigma_1$ was $m$-dimensional the semialgebraic set $B:=\{a\in \sigma_1: \txt{dim}(p_1^{-1}(a)\cap \Sigma_1) = 1\}$,  by Corollary \ref{Hardt}, would be $(m-1)$-dimensional, which would contradict to $(2^{\prime})$. 
\end{proof}
The following elementary lemma is frequently used throughout Section \ref{main}.
\begin{lemma}\label{A:lemma1}
If $X\subset (\K{R}^m,\Vert\cdot\Vert)$ is a scale-invariant semialgebraic variety of dimension $p\leq m$ and $q>0$, then 
\small\begin{align}\label{A:integral}
\int\limits_{a\in X}  \Vert a\Vert^q \,e^{-\frac{\Vert a\Vert^2}{2}} da =  \sqrt{2}^{\,p+q-2}\Gamma\left(\frac{p+q}{2}\right) |X\cap S^{m-1}| =\sqrt{2}^{\,q} \frac{\Gamma\left(\frac{p+q}{2}\right)}{\Gamma\left(\frac{p}{2}\right)} \int\limits_{a\in X}  e^{-\frac{\Vert a\Vert^2}{2}} da, 
\end{align}
where $|X\cap S^{m-1}|$ denotes the volume of the $(r-1)$-dimensional semialgebraic spherical set $X\cap S^{m-1}$.
\end{lemma}
\begin{proof}
By the smooth coarea formula \cite[(A-2)]{Howard} applied to the submersion $\pi: X_{\txt{top}} \rightarrow X_{\txt{top}}\cap S^{m-1}$, $\pi(a)=a/\|a\|$ whose Normal Jacobian is $1/\|a\|^{p-1}$ we have:
\begin{align}
\int\limits_{a\in X} \Vert a\Vert^q \,e^{-\frac{\Vert a\Vert^2}{2}} da = \int\limits_{0}^{+\infty} r^{p+q-1} e^{-\frac{r^2}{2}} dr\ |X\cap S^{m-1}| = \sqrt{2}^{\,p+q-2}\Gamma\left(\frac{p+q}{2}\right) |X\cap S^{m-1}|
\end{align}
Combining this with the same formula for $q=0$ we obtain the second equality in \eqref{A:integral}.
\end{proof}

\bibliographystyle{plain}

\bibliography{references}

\noindent

\bigskip{\footnotesize
\textsc{Universidad de Cantabria, Av. de los Castros, s/n 39005 Santander, Spain}\par
\textit{E-mail address:} \texttt{beltranc@unican.es}
}

\bigskip{\footnotesize
\textsc{SISSA, via Bonomea 265, 34136 Trieste, Italy}\par  
\textit{E-mail address:} \texttt{kkozhasov@sissa.it}
}

\end{document}